\documentclass[a4paper]{amsart}
\makeatletter
\let\cl@chapter\undefined
\makeatletter

\usepackage{amsmath}
\usepackage{amssymb}
\usepackage[maxnames=4,giveninits=true,doi=false,isbn=false,citestyle=numeric,sortcites=true]{biblatex}
\usepackage{dsfont}
\usepackage{enumerate}
\usepackage{hyperref}
\usepackage{mathrsfs}
\usepackage{mathtools}\mathtoolsset{showonlyrefs}
\usepackage[capitalize]{cleveref} 

\makeatletter
\newcommand\definealphabetloop[3]{%
  \ifx\relax#3\expandafter\@gobble\else\expandafter\@firstofone\fi
  {\expandafter\providecommand\expandafter*\csname#1#3\endcsname{#2{#3}}%
   \definealphabetloop{#1}{#2}}%
}%
\newcommand\definealphabet[2]{%
  \definealphabetloop{#1}{#2}abcdefghijklmnopqrstuvwxyzABCDEFGHIJKLMNOPQRSTUVWXYZ\relax
}%
\definealphabet{b}{\mathbb}%
\definealphabet{c}{\mathcal}%
\definealphabet{f}{\mathfrak}%
\definealphabet{r}{\mathrm}%
\definealphabet{s}{\mathscr}%
\definealphabet{t}{\mathtt}%
\makeatother

\crefformat{footnote}{#2\footnotemark[#1]#3}

\allowdisplaybreaks

\newcounter{counter}
\counterwithin{counter}{section}
\theoremstyle{plain}
\newtheorem{theorem}[counter]{Theorem}
\newtheorem{lemma}[counter]{Lemma}
\newtheorem{corollary}[counter]{Corollary}
\theoremstyle{definition}
\newtheorem{example}[counter]{Example}
\newtheorem{definition}[counter]{Definition}
\newtheorem{remark}[counter]{Remark}

\newcommand{\Elem}{\mathcal{E}}

\newcommand{\1}{\mathds{1}}
\newcommand{\Tr}{\mathrm{Tr}}

\begin{document}

\title[Cylindrical stochastic integration]{Cylindrical stochastic integration and applications to financial term structure modeling}
\date{\today}

\author{Johannes Assefa}
\address{Department of Mathematics\\TU Dresden}
\email{johannes.assefa@tu-dresden.de}
\author{Philipp Harms}
\address{Division of Mathematical Sciences\\Nanyang Technological University, Singapore}
\email{philipp.harms@ntu.edu.sg}
\thanks{The authors thank Christa Cuchiero, Claudio Fontana, Johannes M\"uller, Thorsten Schmidt, Guillaume Szulda, and Josef Teichmann for helpful discussions. Philipp Harms gratefully acknowledges financial support by the National Research Foundation Singapore under the award NRF-NRFF13-2021-0012 and by Nanyang Technological University Singapore under the award NAP-SUG}

\date{\today}

\begin{abstract}
We develop a novel---cylindrical---solution concept for stochastic evolution equations. Our motivation is to establish a Heath--Jarrow--Morton framework capable of analysing financial term structures with discontinuities, overcoming deep stochastic-analytic limitations posed by mild or weak solution concepts. Our cylindrical approach, which we investigate in full generality, bypasses these difficulties and nicely mirrors the structure of a large financial market.  
\end{abstract}

\maketitle

\tableofcontents

\section{Introduction}

We develop a new mathematical framework, using cylindrical measure-valued processes, for stochastic modelling of financial term structures with discontinuities, as observed in energy and interest rate markets. 

Cylindrical random variables came up in the 1960s and 1970s in the development of probability theory on Banach spaces and locally convex vector spaces \cite{schwartz1973radon}.
Their raison d'\^etre is that they are easy to construct, for instance by specifying their Fourier transform or by writing down a stochastic integral equation in weak form.
However, it can be notoriously difficult to verify that a cylindrical random variable is non-cylindrical, i.e., a random variable in the usual sense.
This is one of the main challenges in the construction of infinite-dimensional stochastic integrals, and the answer is negative beyond e.g.\@ nuclear spaces \cite{ustunel1985stochastic} or UMD Banach spaces \cite{vanneerven2007stochastic, veraar2016cylindrical}.

We circumvent these difficulties by fully embracing the cylindrical approach.
Thus, we give up on the hard problem of constructing non-cylindrical integrals and instead work with much simpler cylindrical integrals.
Specifically, given any existing integral, stochastic or not, we define a new cylindrical integral by a canonical and simple construction; see \cref{thm:cylindrical}.
In important special cases, this yields the cylindrical Brownian or L\'evy integrals of \cite{berman1983weak, metivier1980stochastic, ondrejat2005brownian, applebaum2010cylindrical}.
Importantly, bounds for the non-cylindrical integral translate into bounds for the induced cylindrical integral. 
This allows us to develop a theory of cylindrical stochastic evolution equations with suitable a-priori estimates; see \cref{thm:cylindrical_stochastic_evolution}.

From a financial perspective, cylindrical processes can be seen as models of large financial markets.
Indeed, a large financial market is simply a collection of tradeable assets, all of which are semimartingales, with no further structure imposed \cite{cuchiero2016new}.
If one assumes without loss of generality that linear combinations of tradeable assets are tradeable, then a large financial market is precisely a cylindrical semimartingale, i.e., a collection of semimartingales indexed by the elements of some linear space.

We exploit this connection to define cylindrical models for financial term structures with discontinuities.
The discontinuities are caused by jumps in the underlying at predictable times and are empirically well documented in real-world markets.
For example, interest rates are influenced by policy updates of the central bank, stock prices are influenced by earnings announcements of the company, energy prices are influenced by maintenance works in the power grid, and all of these events tend to happen at pre-scheduled times; cf. \cref{fig:jumps1,fig:jumps2}.
The immediate modeling implication is that the forward rates cannot be continuous functions of the time to maturity but must be measure-valued \cite{fontana2020term, cuchiero2022measure}.

More precisely, we develop a Heath--Jarrow--Morton (HJM) framework for measure-valued forward rates as in \cite{fontana2020term}.
Actually, the forward rates are merely cylindrically measure-valued because measure-valued stochastic convolutions are ill-defined, as explained in \cref{rem:measure_valued_integration}.
Nevertheless, bond prices and a roll-over bank account are well-defined, as shown in \cref{thm:term_structure}, and the corresponding large financial market satisfies no asymptotic free lunch with vanishing risk \cite{cuchiero2016new} if a HJM drift condition is satisfied under an equivalent probability measure.
Extensions to more general driving noise would be possible within the same framework, thanks to the general construction of the cylindrical stochastic integral.
Our modeling framework is versatile and can be applied to other financial markets, as well, including energy markets. 
In principle, it also lends itself to numerical discretization by similar methods as for non-cylindrical models \cite{harms2019weak, cox2019weak}, but we have not explored this any further.

Markovian methods constitute an interesting complementary approach for constructing measure-valued processes.
They can be used to treat non-Lipschitz coefficients as for instance in super Brownian motion \cite{etheridge2000introduction} and more general affine or polynomial measure-valued processes \cite{cuchiero2019polynomial, cuchiero2021infinite, cuchiero2021measure}.
Notably, these processes are non-cylindrical and take values in the cone of non-negative measures.
There are wide-ranging financial applications in stochastic portfolio theory \cite{cuchiero2019probability}, rough volatility modeling \cite{cuchiero2020generalized}, and recently also term structure modeling \cite{cuchiero2022measure}.
The connection to cylindrical stochastic analysis is via martingale problems with cylindrical test functionals, as explained in \cref{rem:martingale_problems}.

The structure of the paper is as follows. \cref{sec:cylindrical_integration} develops the cylindrical integral in full generality. 
\cref{sec:cylindrical_stochastic} considers the special case of stochastic integration and establishes well-posedness of cylindrical stochastic evolution equations.
\cref{sec:term_structure} develops applications to cylindrically measure-valued forward rate models.
\cref{sec:vector_measure} explains relations to cylindrical vector measures, which are of independent interest but not needed elsewhere in the paper.
\cref{sec:measurable_version} contains some auxiliary results on measurable versions of stochastic processes.

\begin{figure}[h]
\centering
\includegraphics[width=\textwidth]{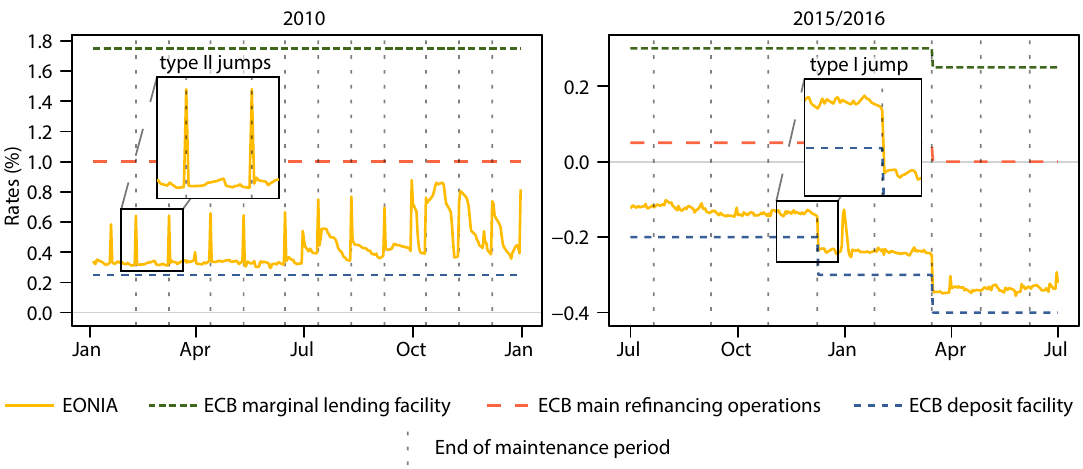}
\caption{EONIA rates jump at predictable times due to pre-scheduled monetary policy meetings of the ECB and due to regulatory constraints coming into effect at pre-scheduled dates. Figure taken from \cite{fontana2020term}.}
\label{fig:jumps1}
\end{figure}

\begin{figure}[h]
\centering
\includegraphics[width=0.8\textwidth]{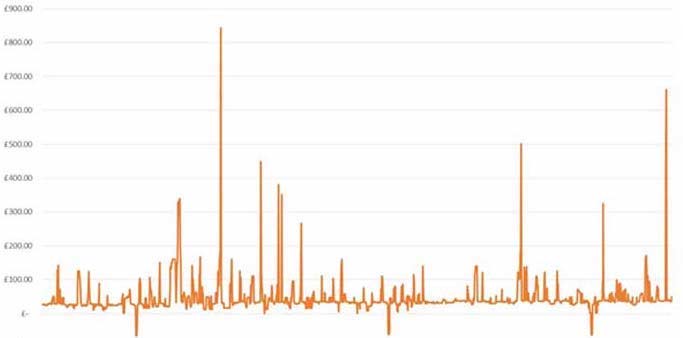}
\caption{Half hourly electricity prices in the UK jumped at predictable times in October 2016 due to pre-scheduled maintenance work on French nuclear power plants. Figure taken from \cite{myice2017what}.}
\label{fig:jumps2}
\end{figure}

\section{Cylindrical vector integration}
\label{sec:cylindrical_integration}

This section describes both abstractly and by examples how any given integral can be turned into a cylindrical integral.
The construction is canonical and simple.
Throughout this section, $E$, $F$, $G$, and $H$ are Banach spaces, and there is a continuous bilinear function $F\times G\to H$, which is denoted by juxtaposition.
Moreover, $S$ is a set, $\cA$ is an algebra of subsets of $S$, and $\Sigma$ is the sigma algebra generated by $\cA$. A finitely additive map $\mu:\cA\to G$ is called vector measure. A function $f:S\to F$ is called elementary if it is measurable and has finite range.
The set of elementary functions $f:S\to F$ is denoted by $\Elem(S,F)$. 
Limits of pointwise converging sequences of elementary functions are called strongly measurable.

\begin{definition}[Elementary integrals]
\label{def:elementary_integral}
The elementary integral of $\mu$ with respect to the bilinear form $F\times G\to H$ is the linear operator
\begin{equation}
\Elem(S,F) \ni f \mapsto \int_S f\mu:=\sum_{i=1}^n f_i\, \mu(A_i) \in H,
\end{equation}
where $n\in\bN$, $f_i \in F$, and $A_i \in \cA$ are such that $f=\sum_{i=1}^n f_i \1_{A_i}$.
The operator norm of this elementary integral, with $\Elem(S,F)$ carrying the supremum norm, is called semivariation of $\mu$ and is denoted by $\|\mu\|_{F,G,H}$.
Moreover, for any $f\in \Elem(S, F)$, the function
\begin{equation}
f\mu:\cA\ni A\mapsto \int_S \1_A f\mu \in H
\end{equation}
is finitely additive, and its semivariation with respect to the canonical bilinear form $\bR\times H \to H$ is denoted by $\|f\|_\mu:=\|f\mu\|_{\bR,H,H}$. 
\end{definition}

\begin{remark}[Notation]
The above notation is quite concise and has been chosen carefully. For greater clarity, we add some remarks.
\begin{enumerate}[(a)]
\item Integration is denoted by $\int f\mu$ instead of $\int fd\mu$. The advantage is that this results in the common notation $\int fdX$ in the special case where $\mu=dX$ is the differential of a semimartingale $X$. This is described in \cref{sec:cylindrical_stochastic}. 
\item The expression $f\mu$ implicitly uses the bilinear form $F\times G\to H$, which by convention is denoted merely by juxtaposition. Accordingly, $f\mu$ is a finitely additive function $\cA\to H$. 
\item The semivariation $\|\mu\|_{F,G,H}$ is defined by integrating $F$-valued integrands over the $G$-valued integrator $\mu$ and measuring the outcome in the norm of $H$. 
\item In the special case $F=\bR$ and $G=H$, the semivariation $\|f\mu\|_{\bR,H,H}$ is defined by integrating $\bR$-valued integrands over the $H$-valued integrator $\mu$ and measuring the outcome in the norm of $H$. This integral uses the canonical bilinear form $\bR\times H\to H$.
\end{enumerate}
\end{remark}

A major challenge in building a good integration theory is to extend the elementary integral to a large set of integrands.
We next list a selection of well-known examples.
All of these are non-cylindrical integrals, which can subsequently be cylindrified with the help of \cref{thm:cylindrical}. 

\begin{example}[Integral on bounded measurable functions]
\label{ex:integral_bounded}
\cite{diestel1977vector}
Assume that $F$ is finite dimensional.
Then, $\Elem(S,F)$ is dense in the set $\cL^\infty(S,F)$ of bounded (strongly) measurable functions with the supremum norm. 
Indeed, bounded subsets of $F$ can be covered by finitely many $\epsilon$-balls, for any $\epsilon>0$, and projection to the nearest center of a ball defines an $\epsilon$-approximation in $\cL^\infty(S,F)$. 
If $\|\mu\|_{F,G,H}<\infty$, then the elementary integral extends uniquely to a continuous linear operator
\begin{equation}
\cL^\infty(S,F) \ni f \mapsto \int_S f\mu \in H.
\end{equation}
A set $A\subset S$ is called a $\mu$-null set if
\begin{equation}
\inf \left\{\sum_{i\in\bN}\|\1_{A_i}\mu\|_{F,G,H}: A_i \in \cA, A\subseteq\bigcup_{i\in\bN}A_i\right\}=0.
\end{equation}
The space $L^\infty(S,F)$ is defined as the quotient of $\cL^\infty(S,F)$ modulo equality up to null sets, and the elementary integral factors through a continuous linear operator
\begin{equation}
L^\infty(S,F) \ni f \mapsto \int_S f\mu \in H.
\end{equation}
\end{example}

The disadvantage of this integral is that the uniform topology on the space of integrands is quite strong. 
A more elaborate version, which comes with a dominated convergence theorem, is presented next.

\begin{example}[Bartle integral]
\label{ex:integral_bartle}
\cite{bartle1956general}
Let $\mu:\Sigma\to G$ be countably additive on $\cA=\Sigma$ with $\|\mu\|_{F,G,H}<\infty$ and suppose that there exists a control measure $\nu$, i.e., a countably additive function $\nu:\Sigma\to \bR_{\geq 0}$ such that $\nu(A)\to 0$ is equivalent to $\|\1_A\mu\|_{F,G,H}\to 0$.
Then, the elementary integral extends uniquely to a continuous linear map
\begin{equation}
\cL^\infty(S,F) \ni f \mapsto \int_S f\mu \in H,
\end{equation}
where $\cL^\infty(S,F)$ carries the structure\footnote{Boundedly pointwise convergence is a non-topological notion of convergence; see e.g. \cite{dolecki2009initiation}.} of boundedly pointwise convergence. As before, one may pass to the quotient space $L^\infty(S,F)$.
\end{example}

The take-away is that existence of a control measure already implies dominated convergence, even in the general context of vector integration. This is extremely useful in many applications, including stochastic integration, but is limited to bounded integrands. The most general class of integrands can be obtained by an isometric extension of the elementary integral, as shown in the subsequent two examples.

\begin{example}[Lebesgue integral]
\label{ex:integral_lebesgue}
Let $F=G=H=\bR$ with bilinear form $\bR\times \bR\to \bR$ given by multiplication. Assume that $\mu:\Sigma\to\bR_{\geq 0}$ is countably additive on $\cA=\Sigma$, i.e., $\mu$ is a measure. Let $L^1(S,\bR)$ be the completion of $\Elem(S,\bR)$ with respect to the seminorm $\|\cdot\|_\mu$.
Then, $L^1(S,\bR)$ can be identified with the Lebesgue space of integrable real-valued functions modulo equality up to null sets. 
Moreover, the elementary integral extends to a linear operator
\begin{equation}
L^1(S,\bR) \ni f \mapsto \int_S f\mu \in \bR
\end{equation}
with operator norm less than or equal to $1$.
\end{example}

\begin{example}[Bochner integral]
\label{ex:integral_bochner}
\cite{bochner1933integration}
Let $F=H$ and $G=\bR$ with bilinear form $F\times\bR\to F$ given by vector-times-scalar multiplication, let $\mu:\Sigma\to\bR_{\geq 0}$ be non-negative and countably additive on $\cA=\Sigma$, and let $L^1(S,F)$ be the completion of $\Elem(S,F)$ with respect to the seminorm $\|\cdot\|_\mu$.
Then, $L^1(S,F)$ can be identified with the Bochner space of strongly measurable functions $f:S\to F$ such that $\|f\|\in L^1(S,\bR)$, modulo equality up to null sets.
Moreover, the elementary integral extends to a linear operator
\begin{equation}
L^1(S,H) \ni f \mapsto \int_S f\mu \in H
\end{equation}
with operator norm less than or equal to $1$.
\end{example}

Some further examples of integrals are presented in \cref{sec:cylindrical_stochastic}.
The common point is that the elementary integral has a continuous extension to some Banach space of integrands, which contains the elementary integrands as a dense subspace. 
We next show that any such integral can be turned into a new cylindrical integral in a canonical and general way.

\begin{theorem}[Cylindrical integration]
\label{thm:cylindrical}
Assume that the elementary integral in \cref{def:elementary_integral} has a continuous extension to some normed space $I$ of integrands:
\begin{equation}
I \ni f \mapsto \int_S f\mu \in H.
\end{equation}
If $E$ is non-trivial, then the linear operator
\begin{equation}
L(E,I) \ni f \mapsto \Big(E\ni e \mapsto \int_S f(e)\mu \in H\Big) \in L(E,H)
\end{equation}
has the same operator norm as the original integral and is called cylindrified integral.
\end{theorem}

\begin{proof}
As $E$ is non-trivial, the functor $L(E,\cdot)$ is norm-preserving. To spell this out in further detail, let $T:I\to H$ be the given integral, and let 
	\begin{equation*}
		\hat T:L(E,I)\to L(E,H),
		\qquad
		f\mapsto\left(e\mapsto T(fe)\right)
	\end{equation*}
	be the cylindrified integral, obtained by applying the functor $L(E,\cdot)$ to $T$. Writing $\bigcirc$ for the unit balls in the respective spaces, we have 
\begin{align*}
\|\hat T\|_{L(L(E,I),L(E,H))}
&=
\sup_{f \in \bigcirc L(E,I)} 
\|\hat Tf\|_{L(E,H)}
=
\sup_{f \in \bigcirc L(E,I)} 
\sup_{e \in \bigcirc E}
\|(\hat Tf)e\|_{H}
\\&=
\sup_{f \in \bigcirc L(E,I)} 
\sup_{e \in \bigcirc E}
\|T(fe)\|_{H}
\leq
\sup_{f \in \bigcirc I} 
\|Tf\|_{H}
=
\|T\|_{L(I,H)},
\end{align*}
where the inequality is actually an equality because $E$ is non-trivial.
\end{proof}

The take-away is that any integral can be turned into a cylindrical integral, which is an element-wise application of the original integral. 
As an aside, the name cylindrical integral stems from stochastic integration: There, $H$ is a space of random variables, and elements of $L(E,H)$ are called cylindrical random variables; see \cref{sec:cylindrical_stochastic}.
Cylindrification can be expressed also in the language of vector measures, as explained in \cref{sec:vector_measure}.

\section{Cylindrical stochastic analysis}
\label{sec:cylindrical_stochastic}

This section describes some stochastic integrals together with their cylindrified counterparts and establishes well-posedness of cylindrical stochastic evolution equations under Lipschitz conditions on the coefficients.
Throughout this section,
$(\Omega,\cF,\bF=(\cF_t)_{t\in[0,T]},\bP)$ is a filtered probability space satisfying the usual conditions,
$\Omega_T=[0,T]\times\Omega$,
$\cR$ is the semi-ring of predictable rectangles in $\Omega_T$, i.e.,
\begin{equation}
\cR=\big\{\{0\}\times A:A\in\cF_0\big\} \cup \big\{(s,t]\times A: A \in \cF_s, 0\leq s\leq t\leq T\big\},
\end{equation}
$\cA$ is the algebra generated by $\cR$, and $\cP$ is the sigma algebra generated by $\cA$.
As any set in $\cA$ is a finite union of sets in $\cR$, the elementary functions over $\cR$ and $\cA$ coincide and are denoted simply by $\Elem(\Omega_T,\bR)$.
Unless specified otherwise, $\Omega_T$ carries the predictable sigma algebra $\cP$.
Subscripts $\bF$ in names of functions spaces denote adaptedness to the filtration $\bF$, examples being $C_\bF([0,T],L^2(\Omega,H))$ or $L^2_\bF(\Omega,L^2([0,T],H))$.
The space of Hilbert-Schmidt operators between Hilbert spaces $U$ and $H$ is denoted by $L_2(U,H)$.

\begin{example}[Semimartingale integral]
\label{ex:semimartingale}
\cite{kwapien1992random,bichteler2002stochastic,dinculeanu2000vector,metivier1980stochastic}
Let $p \in [1,\infty)$,
let $X:[0,T]\to L^p(\Omega,\bR)$ be a stochastic process which is adapted and right-continuous in probability,
and let $dX:\cR\to L^p(\Omega,\bR)$ be given by
\begin{align}
dX(\{0\}\times A)&=X_0\1_A, && A \in \cF_0,
\\
dX((s,t]\times A)&=(X_t-X_s)\1_A,&& A \in \cF_s, 0\leq s\leq t\leq T.
\end{align}
Then, $dX$ extends uniquely to an additive function $dX:\cA\to L^p(\Omega,\bR)$ because any set $\cA$ is a finite union of sets in $\cR$.
Moreover, the following are equivalent:
\begin{enumerate}[(a)]
\item\label{ex:semimartingale_a} $dX$ extends (uniquely) to a countably additive function $dX:\cP\to L^p(\Omega,\bR)$.
\item\label{ex:semimartingale_b} The elementary integral
\begin{equation}
\Elem(\Omega_T,\bR)
\ni
f
\mapsto
\int_{\Omega_T}fdX
\in L^p(\Omega,\bR)
\end{equation}
is continuous with respect to the supremum norm on $\Elem(\Omega_T,\bR)$ and therefore has a unique continuous extension\footnote{The set $\cE(\Omega_T,\bR)$ is dense in $\cL^\infty(\Omega_T,\bR)$, as shown in \cref{ex:integral_bounded}.}
\begin{equation}
\cL^\infty(\Omega_T,\bR)
\ni
f
\mapsto
\int_{\Omega_T}fdX
\in L^p(\Omega,\bR),
\end{equation}
whose operator norm $\|dX\|_{\bR,L^p(\Omega),L^p(\Omega)}$ is the semivariation of the vector measure $dX$ and is called $L^p$-Emery metric.
\item\label{ex:semimartingale_c} Up to a modification, $X=M+A$ for a $p$-integrable martingale $M$ and a process $A$ of $p$-integrable variation.
\end{enumerate}
As usual, one may pass from $\cL^\infty(\Omega_T,\bR)$ to its quotient $L^\infty(\Omega_T,\bR)$. 
Moreover, for any Banach space $E$, the resulting cylindrified integral is of the form
\begin{equation}
L(E,L^\infty(\Omega_T,\bR))
\ni
f
\mapsto
\int_{\Omega_T}fdX
\in L(E,L^p(\Omega,\bR)).
\end{equation}
\end{example}

\begin{example}[It\^o integral]
\label{ex:ito}
\cite{vanneerven2007stochastic,cox2011vector}
Let $H$ and $U$ be separable Hilbert spaces, and let $W$ be $U$-cylindrical Brownian motion, i.e., $W:L^2([0,T],U)\to L^2(\Omega,\bR)$ is a linear isometry with values in centered Gaussian random variables.
Then, the elementary integral
\begin{equation}
\Elem(\Omega_T,L_2(U,H)) \ni f \mapsto \int_{\Omega_T} fdW \in L^2(\Omega,H)
\end{equation}
is well defined and extends to the It\^o isometry
\begin{equation}
L^2(\Omega_T,L_2(U,H))=L^2_{\bF}(\Omega,L^2([0,T],L_2(U,H))) \ni f \mapsto \int_{\Omega_T} fdW \in L^2(\Omega,H).
\end{equation}
The equality above is due to the fact that $[0,T]$ carries an absolutely continuous measure \cite{rudiger2012isomorphism}.
More generally, if $H$ is merely UMD Banach and $p \in (1,\infty)$, the elementary integral extends to an isomorphism
\begin{equation}
L^p_{\bF}(\Omega,\gamma(L^2([0,T],U)),H) \ni f \mapsto \int_{\Omega_T} fdW \in L^p(\Omega,H),
\end{equation}
where $\gamma$ stands for $\gamma$-radonifying linear operators.
Even more generally, if $H$ is merely an abstract $L^1$ space, one retains a continuous (but not necessarily isomorphic) integral of the above form.
However, beyond these settings, there are Banach spaces $H$ such that the above extension of the elementary integral does not exist.
Nevertheless, by \cref{thm:cylindrical} it always exists in a cylindrical form, namely, as a linear isomorphism
\begin{equation}
L(E,L^p_{\bF}(\Omega,L^2([0,T],U)) \ni f \mapsto \int_{\Omega_T} fdW \in L(E,L^p(\Omega,\bR)).
\end{equation}
This generalizes the above integrals if $E$ is a dual or pre-dual of $H$ because $L^p(\Omega,H)$ is then continuously included in $L(E,L^p(\Omega,\bR))$.
\end{example}

\begin{example}[Stochastic convolutions]
\label{ex:convolution}
Let $H$ and $U$ be separable Hilbert spaces, let $W$ be $U$-cylindrical Brownian motion, let $E$ be a Banach space, and let $S:\bR_{\geq 0}\to L(E)$ be a strongly continuous semigroup.
For any $f \in L(E,L^2(\Omega_T,L_2(U,H)))$, the function
\begin{equation}
E\times[0,T] \ni (e,t)\mapsto \big((s,\omega)\mapsto \1_{[0,t]}(s) f(S_{t-s}e)(s,\omega)\big) \in L^2(\Omega_T,L_2(U,H))
\end{equation}
is continuous.
Applying the It\^o integral of \cref{ex:ito} yields a continuous function
\begin{equation}
E\times[0,T] \ni (e,t)\mapsto \int_0^t f(S_{t-s}e)(s) dW_s \in L^2(\Omega,H).
\end{equation}
The integral is continuous and adapted in $t\in[0,T]$ and continuous and linear in $e \in E$.
Thus, writing $f(S_{t-s}e)$ as $(S_{t-s}^*f)e$, which is a pull-back of $f$ along $S_{t-s}$ followed by evaluation at $e$, one obtains a continuous linear operator
\begin{equation}
L(E,L^2(\Omega_T,L_2(U,H)))\ni f \mapsto \int_0^\cdot S_{\cdot-s}^*f_s dW_s \in L(E,C_\bF([0,T],L^2(\Omega,H))),
\end{equation}
which shall be called cylindrical stochastic convolution.
Similarly, using pathwise Bochner integrals instead of It\^o integrals, one obtains a cylindrical deterministic convolution
\begin{equation}
L(E,L^2(\Omega_T,H))\ni f \mapsto \int_0^\cdot S_{\cdot-s}^*f_s ds \in L(E,C_\bF([0,T],L^2(\Omega,H))).
\end{equation}
\end{example}

The cylindrical convolutions in \cref{ex:convolution} can be used to develop a theory of cylindrical stochastic evolution equations, as shown next. 

\begin{theorem}[Cylindrical stochastic evolution equations]
\label{thm:cylindrical_stochastic_evolution}
Let $U$ and $H$ be separable Hilbert spaces,
let $W$ be $U$-cylindrical Brownian motion,
let $E$ be a Banach space,
let $S:\bR_{\geq 0}\to L(E)$ be a strongly continuous semigroup,
let $X_0 \in L(E,L^2((\Omega,\cF_0,\bP),H))$,
and let $F$ and $G$ be Lipschitz functions
\begin{align}
F:L(E,C_\bF([0,T],L^2(\Omega,H)))&\to L(E,C_\bF([0,T],L^2(\Omega,H))),
\\
G:L(E,C_\bF([0,T],L^2(\Omega,H)))&\to L(E,C_\bF([0,T],L^2(\Omega,L_2(U,H)))).
\end{align}
Then, there exists a unique $X \in L(E,C_\bF([0,T],L^2(\Omega,H)))$ which satisfies
\begin{equation}
X_t = S_t^*X_0 + \int_0^t S_{t-s}^*F(X)_s ds + \int_0^t S_{t-s}^*G(X)_s dW_s,
\end{equation}
for all $t \in [0,T]$, where the integrals are $E$-cylindrical, $S_{t-s}^*F$ denotes the pull-back of $F$ along $S_{t-s}$, and similarly for $S_{t-s}^*G$.
\end{theorem}

\begin{proof}
We first verify that the right-hand side of the integral equation is well-defined.
For any $X \in L(E,C_\bF([0,T],L^2(\Omega,H)))$, the integrands are given by
\begin{align}
F(X) &\in L(E,C_\bF([0,T],L^2(\Omega,H))),
&
G(X) &\in L(E,C_\bF([0,T],L^2(\Omega,L_2(U,H)))).
\end{align}
Any continuous adapted function with values in $L^2(\Omega,H)$ or $L^2(\Omega,L_2(U,H))$ has a predictable version by \cref{lem:predictable}, giving rise to continuous inclusions
\begin{align}
C_\bF([0,T],L^2(\Omega,H))&\to L^2(\Omega_T,H),
\\
C_\bF([0,T],L^2(\Omega,L_2(U,H)))&\to L^2(\Omega_T,L_2(U,H)).
\end{align}
For later use, we note that the operator norm of these inclusions is $\sqrt{T}$.
Via these inclusions, the integrands are given by
\begin{align}
F(X) &\in L(E,L^2(\Omega_T,H)),
&
G(X) &\in L(E,L^2(\Omega_T,L_2(U,H))).
\end{align}
By an application of the cylindrical stochastic convolutions in \cref{ex:convolution}, one obtains that the right-hand side of the integral equation,
\begin{align}
K(X):=S_\cdot^*X_0+\int_0^\cdot S_{\cdot-s}^*F(X)_s ds+\int_0^\cdot S_{\cdot-s}^*G(X)_s dW_s
\end{align}
belongs to $L(E,C_\bF([0,T],L^2(\Omega,H)))$.
The function
\begin{equation}
K: L(E,C_\bF([0,T],L^2(\Omega,H))) \to L(E,C_\bF([0,T],L^2(\Omega,H)))
\end{equation}
is Lipschitz continuous as a composition of Lipschitz functions.
The Lipschitz constant of $K$ is bounded by a constant times $\sqrt{T}$.
Thus, for sufficiently small $T$, the operator $K$ is a contraction and has a unique fixed point by the Banach fixed point theorem.
For arbitrary $T$, one obtains a unique fixed point by concatenation of solutions on small sub-intervals of $[0,T]$.
\end{proof}

\section{Cylindrically measure-valued term structure models}
\label{sec:term_structure}

This section is devoted to cylindrical random measures and applications to financial term structure models.
Throughout this section,
$T^*\in\bR_{>0}$,
$(\Omega,\cF,\bF=(\cF_t)_{t\in[0,T^*]},\bP)$ is a filtered probability space satisfying the usual conditions,
$\Omega_{T^*}=[0,T^*]\times\Omega$ with the predictable sigma algebra $\cP$,
$W$ is $U$-cylindrical Brownian motion on a separable Hilbert space $U$,
$\cB(\bR_{>0})$ is the Borel sigma algebra on the positive half-line $\bR_{>0}=(0,\infty)$,
and $C_0(\bR_{>0})$ is the Banach space of continuous functions $f:\bR_{>0}\to \bR$ with $\lim_{x\to 0}f(x)=\lim_{x\to\infty}f(x) = 0$.
By the Riesz--Markov--Kakutani representation theorem, $C_0(\bR_{>0})^*$ is the space $\cM(\bR_{>0})$ of signed measures $\mu:\cB(\bR_{>0})\to\bR$ with finite total variation $\|\mu\|_{C_0(\bR_{>0})^*}=\|\mu\|_{\bR,\bR,\bR}$.
Similarly, for any Banach space $E$, $\cM(\bR_{>0},E)$ denotes the set of countably additive vector measures $\mu:\cB(\bR_{>0})\to E$ with finite semivariation $\|\mu\|_{\bR,E,E}$.

The origin of this paper lies in the difficulty of constructing measure-valued stochastic integrals, as needed for Heath--Jarrow--Morton models of discontinuous term structures. This difficulty is explained next, and a solution via cylindrical random measures is developed subsequently.

\begin{remark}[Measure-valued stochastic integration]
\label{rem:measure_valued_integration}
The development of measure-valued stochastic integrals meets several difficulties.
As the space $\cM(\bR_{>0})$ is not UMD, there is no It\^o isometry in the sense of \cref{ex:ito}.
However, $\cM(\bR_{>0})$ is an abstract $L^1$ space, i.e., a Banach lattice such that $\|x+y\|=\|x\|+\|y\|$ for any non-negative elements $x$ and $y$ whose infimum $x\wedge y$ vanishes.
Thus, it is lattice isometric to some $L^1$ space \cite[Theorem~4.27]{aliprantis2006positive}.
Therefore, a decoupling inequality holds \cite[Corollary~3.21]{cox2012stochastic}, which implies that the elementary It\^o integral extends to a continuous linear map \cite[Theorem~3.29]{cox2012stochastic}
\begin{equation}
L^2_\bF(\Omega,\gamma(L^2([0,T],U),\cM(\bR_{>0})) \ni f \mapsto \int_0^T fdW \in L^2(\Omega,\cM(\bR_{>0})),
\end{equation}
where $\gamma$ denotes the space of gamma-radonifying operators.
This integral can be used to build measure-valued term structure models which are parameterized by time-of-maturity.
However, it is ill-suited for parameterization by time-to-maturity, as in Heath--Jarrow--Morton models, because measure-valued stochastic convolutions involving the shift semigroup are ill-defined.
For example, if $W$ is scalar Brownian motion, $\delta_0 \in \cM(\bR_{>0})$ is the Dirac measure, and $S$ is the semigroup of left-shifts of measures, then the integral $\int_0^t S_{t-s}\delta_0dW_s$ is almost never a measure because it is the distributional derivative of a Brownian path: indeed, for any $f \in C^\infty_c((0,t))$,
\begin{equation}
\int_0^t S_{t-s}\delta_0(f)dW_s = \int_0^t f(t-s) dW_s = \int_0^t f'(t-s) W_s ds.
\end{equation}
Abstractly, the problem is that the shift semigroup on $\cM(\bR_{>0})$ is not strongly measurable and not Rademacher-bounded.
Replacing the norm topology on $\cM(\bR_{>0})$ by the weak-* topology does not help because the stochastic integral in the above example is not a measure-valued random variable despite the integrand $s\mapsto S_{t-s}\delta_0$ being bounded predictable.
The integral is, however, a cylindrical measure-valued random variable, i.e., an element of $L(C_0(\bR_{>0}),L^2(\Omega))$.
This is detailed in the sequel.
\end{remark}

The following lemma \eqref{lem:random_measures_a} describes cylindrical random measures as a generalization of random measures and \eqref{lem:random_measures_b} explains their relation to $L^p(\Omega)$-valued vector measures.
An important consequence of \eqref{lem:random_measures_b} is that cylindrical vector measures can be integrated not only against $C_0$ functions, as suggested by their definition, but also against bounded Borel-measurable functions.

\begin{lemma}[Cylindrical random measures]
\label{lem:random_measures}
\begin{enumerate}[(a)]
\item \label{lem:random_measures_a}
For any $p\in[1,\infty]$, there is a continuous inclusion
\begin{equation}
L^p(\Omega,\cM(\bR_{>0})) \ni \mu \mapsto \Big(f\mapsto \int f\mu\Big) \in L(C_0(\bR_{>0}),L^p(\Omega)).
\end{equation}
\item \label{lem:random_measures_b}
For any $p\in(1,\infty)$, there is an isometric isomorphism
\begin{equation}
\cM(\bR_{>0},L^p(\Omega)) \ni \mu \mapsto \Big(f\mapsto \int f\mu\Big) \in L(C_0(\bR_{>0}),L^p(\Omega)).
\end{equation}
\end{enumerate}
\end{lemma}

\begin{proof}
\begin{enumerate}[(a)]
\item Let $i$ denote the function in \eqref{lem:random_measures_a}.
To show that $i$ is injective, let $\mu \in L^p(\Omega,\cM(\bR_{>0}))$ with $i(\mu)=0$.
Then, the support of $\mu$ is $\{0\}$ because for any non-zero $\nu\in\cM(\bR_{>0})$, there is $f \in C_0(\bR_{>0})$ with $\int f\nu\neq 0$, and then $U_f=\{\tilde \nu \in \cM(\bR_{>0}):\int f\tilde\nu\neq 0\}$ is an open neighborhood of $\nu$ which satisfies $\bP[\mu\in U_f]=\bP[\int f\mu\neq 0]=0$.
Therefore, $\mu$ vanishes, and $i$ is injective.
Furthermore, $i$ is continuous thanks to the following inequality, which holds for all $\mu \in L^p(\Omega,\cM(\bR_{>0}))$ and $f\in C_0(\bR_{>0})$:
\begin{equation}
\Big\|\int f\mu\Big\|_{L^p(\Omega)}
\leq
\big\|\|f\|_{C_0(\bR_{>0})}\|\mu\|_{\cM(\bR_{>0})}\big\|_{L^p(\Omega)}
=
\|f\|_{C_0(\bR_{>0})} \|\mu\|_{L^p(\Omega,\cM(\bR_{>0}))}.
\end{equation}
\item Let $i$ denote the function in \eqref{lem:random_measures_b}. Then, $i$ is well defined because the elementary integral of any $\mu\in \cM(\bR_{>0},L^p(\Omega))$ extends isometrically to a continuous linear operator $\cL^\infty(\bR_{>0})\to L^p(\Omega)$, as explained in \cref{ex:integral_bounded}, and then restricts to a continuous linear operator $i(\mu):C_0(\bR_{>0})\to L^p(\Omega)$.
The semivariation of $\mu$ is the operator norm of the elementary integral of $\mu$ and coincides with the operator norms of the above extension and restriction.
Thus, $i$ is an isometry.
The surjectivity of $i$ follows from a version of the Riesz representation theorem; cf.  \cite[Section~VI.2]{diestel1977vector}.
Thanks to the reflexivity of $L^p(\Omega)$, the bi-dual of any $T \in L(C_0(\bR_{>0}),L^p(\Omega))$ is an extension $T^{**}:C_0(\bR_{>0})^{**}\to L^p(\Omega)$ of $T$.
Let $B(\bR_{>0})$ be the closure of $\Elem(\bR_{>0})$ in the uniform norm.
As $B(\bR_{>0})$ is isometrically included in $C_0(\bR_{>0})^{**}$ and contains all indicators of Borel sets, one can define $\mu:\cB(\bR_{>0})\ni A \mapsto T^{**}\1_A \in L^p(\Omega)$.
As $T^{**}$ is weak-* to weak-* continuous, $\mu$ is weakly countably additive and, by the Orlicz--Pettis theorem \cite[Corollary I.4.4]{diestel1977vector}, countably additive.
Moreover, $i(\mu)=T$ because the two operators coincide on the dense set of elementary integrands.
Thus, $i$ is surjective.
\end{enumerate}
\end{proof}

\begin{remark}[Countable additivity]
\label{rem:countable_additivity}
Let $p\in(1,\infty)$.
Then, any cylindrical random measure $\mu \in L(C_0(\bR_{>0}),L^p(\Omega))$ is `almost' countably additive in the following sense: for any sequence of disjoint sets $A_i \in \cB(\bR_{>0})$,
\begin{equation}
\mu\Big(\bigcup_i A_i\Big) = \sum_i \mu(A_i)
\end{equation}
up to a null set, which may depend on the sequence $A_i$.
If $\mu$ is non-cylindrical, i.e., $\mu \in L^p(\Omega,\cM(\bR_{>0}))$, the null set can be chosen independently of the sequence $A_i$.
This difference is minor in financial applications.
\end{remark}

Cylindrically measure-valued processes can be constructed by solving cylindrical stochastic evolution equations, as shown next.
The corresponding theory has been developed in \cref{thm:cylindrical_stochastic_evolution}.
For financial applications, the two most important cases are where the semigroup acts via right-shifts on $C_0(\bR_{>0})$ and where the semigroup acts trivially as the identity on $C_0(\bR_{>0})$.
Note that right-shifts on $C_0(\bR_{>0})$ correspond to left-shifts on the dual space $\cM(\bR_{>0})$, in line with the Heath--Jarrow--Morton framework.

\begin{corollary}[Cylindrically measure-valued evolution equations]
\label{cor:cylindrically_evolution}
Let $S:\bR_{\geq 0}\to L(C_0(\bR_{>0}))$ be a strongly continuous semigroup,
let $X_0 \in L(C_0(\bR_{>0}),L^2(\Omega,\cF_0,\bP))$,
and let $F$ and $B$ be Lipschitz functions
\begin{align}
F:L(C_0(\bR_{>0}),C_\bF([0,T],L^2(\Omega)))&\to L(C_0(\bR_{>0}),C_\bF([0,T],L^2(\Omega))),
\\
G:L(C_0(\bR_{>0}),C_\bF([0,T],L^2(\Omega)))&\to L(C_0(\bR_{>0}),C_\bF([0,T],L^2(\Omega,U))).
\end{align}
Then, there exists a unique $X \in L(C_0(\bR_{>0}),C_\bF([0,T],L^2(\Omega)))$ which satisfies for all $t \in [0,T]$ that
\begin{equation}
X_t = S_t^*X_0 + \int_0^t S_{t-s}^*F(X)_s ds + \int_0^t S_{t-s}^*G(X)_s dW_s.
\end{equation}
\end{corollary}

\begin{proof}
This is a special case of \cref{thm:cylindrical_stochastic_evolution}.
\end{proof}

We next describe financial term structure models where the collection of forward rates is given by a cylindrically measure-valued process, as given by \cref{cor:cylindrically_evolution} above.
Then, the prices of bonds or futures are obtained as usual by integration of maturities and exponentiation, i.e., by the application of cylindrical functionals.
This yields a large financial market, whose tradeable assets are cylindrical functionals of a cylindrical process.
Besides the minimal assumption that linear combinations of tradeable assets are tradeable, there is little extra structure imposed.
In particular, it is not assumed that the tradeable assets constitute a non-cylindrical process, and this assumption would also be lacking an economic foundation.

As a proof of concept, we present in \cref{thm:term_structure} below a term structure model with two components: one for interest rates, parameterized by time-to-maturity, and another one for energy prices, parameterized by time-of-maturity.
The motivation is that both interest rate and energy markets exhibit jumps at predictable times, which force the forward rates to be measure-valued.
Moreover, the parameterization by time-to-maturity is well suited for dynamics depending foremost on market structure, as in the case of interest rates, whereas the parameterization by time-of-maturity is well suited for dynamics depending strongly on seasonalities, as in the case of energy markets.
The additive structure of energy forwards corresponds to advance settlement of future contracts and is broken for continuous or terminal settlement unless the interest rate is constant \cite[Section~4.1]{benth2008stochastic}.
We do not know how to model continuously or terminally settled contracts with stochastic interest rates in the present generality; cf. \cite{benth2019mean}.

\begin{theorem}[Cylindrical term structure models]
\label{thm:term_structure}
Let $S:\bR_{\geq 0}\to L(C_0(\bR_{>0}))$ be the semigroup of right shifts,
let $F,\tilde F \in L(C_0(\bR_{>0}),C_\bF([0,T^*],L^2(\Omega)))$,
let $G,\tilde G \in L(C_0(\bR_{>0}),C_\bF([0,T^*],L^2(\Omega,U)))$,
and let $X,\tilde X \in L(C_0(\bR_{>0}),C_\bF([0,T^*],L^2(\Omega)))$ satisfy for all $t\in[0,T^*]$ that
\begin{align}
X_t &= S_t^*X_0 + \int_0^t S_{t-s}^*F_s ds + \int_0^t S_{t-s}^*G_s dW_s,
\\
\tilde X_t &= \tilde X_0 + \int_0^t \tilde F_s ds + \int_0^t \tilde G_s dW_s.
\end{align}
Then, the following statements hold, with operators on $C_0(\bR_{>0})$ viewed as vector measures using the Riesz representation in  \cref{lem:random_measures}~\eqref{lem:random_measures_b}:
\begin{enumerate}[(a)]
\item\label{thm:term_structure_a} For every $T,T_1,T_2 \in [0,T^*]$, the bond price and energy futures processes
\begin{equation}
P(t,T) = \exp\big(-X_t(0,T-t]\big),
\qquad
\tilde P(t,T_1,T_2) = \tilde X_t(T_1,T_2),
\qquad
t\in[0,T^*],
\end{equation}
have a predictable version, which is unique up to $dt\otimes\bP$-null sets.
\item\label{thm:term_structure_b} The roll-over bank account process
\begin{equation}
B_t = \lim_{n\to \infty} \prod_{i=1}^n P(t(i-1)/n,ti/n)^{-1},
\qquad
t \in [0,T^*],
\end{equation}
exists, has a predictable version, and is given by
\begin{align}
B_t &= \exp\Big(X_0(0,t]+\int_0^t F_s(0,t-s] ds + \int_0^t G_s(0,t-s]dW_s\Big),
&&
t\in[0,T^*].
\end{align}
\item\label{thm:term_structure_c} For every $T,T_1,T_2\in[0,T^*]$, the discounted bond price $P(t,T)/B_t$ and the discounted energy future $\tilde P(t,T_1,T_2)/B_t$ are local martingales if and only if the following drift conditions hold:
\begin{equation}
F_t(0,T-t]=\frac12 \|G_t(0,T-t]\|_U^2,
\qquad
\tilde F_t(T_1,T_2)=0,
\qquad
dt\otimes\bP\text{-almost surely.}
\end{equation}
\item\label{thm:term_structure_e} If the drift condition \eqref{thm:term_structure_c} is valid for all $T,T_1,T_2 \in [0,T^*]$, then there is no asymptotic free lunch with vanishing risk under any probability measure equivalent to $\bP$ in the large financial market whose primary traded asssets are the discounted bond prices $P(t,T)/B_t$ together with the discounted energy futures $\tilde P(t,T_1,T_2)/B_t$, for $T,T_2,T_2\in[0,T^*]$.
\end{enumerate}
\end{theorem}

\begin{proof}
\begin{enumerate}[(a)]
\item Thanks to the continuous linear inclusion \cite{tappe2010note}
\begin{equation}
C_\bF([0,T^*],L^2(\Omega))\to L^2(\Omega_{T^*}),
\end{equation}
which amounts to the choice of a predictable version, the process $X$ is a continuous linear operator
\begin{equation}
X : C_0(\bR_{>0})\to L^2(\Omega_{T^*}).
\end{equation}
Then, by the same argument as in \cref{lem:random_measures}~\eqref{lem:random_measures_b}, $X$ has a Riesz representation as a countably additive vector measure
\begin{equation}
X : \cB(\bR_{>0})\to L^2(\Omega_{T^*}).
\end{equation}
Thanks to the countable additivity, the function
\begin{equation}
\bR \ni x \mapsto X(0,x] \in L^2(\Omega_{T^*})
\end{equation}
is right-continuous with left limits.
By \cref{lem:measurable}, it has a measurable version
\begin{equation}
\bR\times\Omega_{T^*}\ni (x,t,\omega) \mapsto X(0,x](t,\omega) \in \bR.
\end{equation}
As for any $T\in[0,T^*]$, the function
\begin{equation}
\Omega_{T^*}\ni (t,\omega)\mapsto (T-t,t,\omega)\in\bR\times\Omega_{T^*}
\end{equation}
is measurable, one obtains the measurability of the function
\begin{equation}
\Omega_{T^*}\ni (t,\omega) \mapsto X(0,T-t](t,\omega) \in \bR.
\end{equation}
This function is uniquely determined up to $dt\otimes\bP$-null sets because, for any $\tilde X$ which originates from another choice of measurable version above,
\begin{equation}
\int_0^{T^*}\bE\Big[\big(X(0,T-t](t,\omega)-\tilde X(0,T-t](t,\omega)\big)^2\Big]dt=0.
\end{equation}
Thus, by exponentiation, one obtains that the bond price process $P(\cdot,T)$ has a predictable version, which is determined uniquely up to a $dt\otimes\bP$-null set.
A similar reasoning applies to the energy futures process $\tilde P(\cdot,T_1,T_2)$.
\item
For the purpose of the proof, let $B_t$ be defined for any $t \in [0,T^*]$ as
\begin{equation}
B_t = \exp\Big(X_0(0,t]+\int_0^t F_s(0,t-s] ds + \int_0^t G_s(0,t-s]dW_s\Big).
\end{equation}
These integrals are well-defined because for any $t \in [0,T^*]$, the integrands have predictable versions
\begin{align}
\Omega_{T^*} \ni (s,\omega)&\mapsto F(0,t-s](s,\omega) \in \bR,
\\
\Omega_{T^*} \ni (s,\omega)&\mapsto G(0,t-s](s,\omega) \in U,
\end{align}
which are uniquely determined up to $ds\otimes\bP$-null sets and square-integrable,
as can be seen similarly to \eqref{thm:term_structure_a}.
To define the roll-over bank account process, let $t \in [0,T^*]$, $n\in\bN$, $t_i=ti/n$, and
\begin{equation}
B^n_t
=
\prod_{i=1}^n P(t_{i-1},t_i)^{-1}
=
\exp \Big(\sum_{i=1}^n X_{t_{i-1}}(0,t_i-t_{i-1}]\Big).
\end{equation}
Then, by the evolution equation for $X$,
\begin{align}
\log B^n_t
&=
\sum_{i=1}^n S_{t_{i-1}}^* X_0(0,t_i-t_{i-1}]
+
\sum_{i=1}^n \int_0^{t_{i-1}} S_{t_{i-1}-s}^* F_s(0,t_i-t_{i-1}] ds
\\&\qquad+
\sum_{i=1}^n \int_0^{t_{i-1}} S_{t_{i-1}-s}^* G_s(0,t_i-t_{i-1}] dWs
\\&=
X_0(0,t_n]
+
\sum_{i=1}^n \int_0^{t_{i-1}} F_s(t_{i-1}-s,t_i-s] ds
\\&\qquad+
\sum_{i=1}^n \int_0^{t_{i-1}} G_s(t_{i-1}-s,t_i-s] dW_s.
\end{align}
In terms of the round-up-to-the-grid operator
\begin{equation}
\lceil s\rceil_n = \min\{t_i: i\in\bN,t_i\geq s\},
\qquad
s \in [0,t],
\end{equation}
this can be rewritten by purely algebraic manipulations as
\begin{align}
\log B^n_t
&=
X_0(0,t]
+
\int_0^t F_s(\lceil s\rceil_n-s,t-s] ds
+
\int_0^t G_s(\lceil s\rceil_n-s,t-s] dW_s.
\end{align}
For large $n$, $\lceil s\rceil_n-s$ tends to $0$ from above.
As $F$ and $G$ are countably additive vector measures on $\cB(\bR_{>0})$, it follows that
\begin{align}
\lim_{n\to\infty}\|F_s(\lceil s\rceil_n-s,t-s]-F_s(0,t-s]\|_{L^2(\Omega)}&=0,
\\
\lim_{n\to\infty}\|G_s(\lceil s\rceil_n-s,t-s]-G_s(0,t-s]\|_{L^2(\Omega,U)}&=0.
\end{align}
Moreover, one has the uniform bounds
\begin{align}
\|F_s(\lceil s\rceil_n-s,t-s]\|_{L^2(\Omega)}
&\leq
\|F\|_{L(C_0(\bR_{>0}),C_\bF([0,T^*],L^2(\Omega)))},
\\
\|G_s(\lceil s\rceil_n-s,t-s]\|_{L^2(\Omega,U)}
&\leq
\|G\|_{L(C_0(\bR_{>0}),C_\bF([0,T^*],L^2(\Omega,U)))}.
\end{align}
Thus, by the dominated convergence theorem,
\begin{align}
\lim_{n\to\infty} \int_0^t \|F_s(\lceil s\rceil_n-s,t-s]-F_s(0,t-s]\|_{L^2(\Omega)}^2 ds =0,
\\
\lim_{n\to\infty} \int_0^t \|G_s(\lceil s\rceil_n-s,t-s]-G_s(0,t-s]\|_{L^2(\Omega,U)}^2 ds =0.
\end{align}
This ensures convergence of the integrals in the above formula for $\log B^n_t$.
Thus, we have shown that $B^n_t$ converges to $B_t$ as $n$ tends to infinity.

\item This follows along the lines of \cite[p.~61ff]{Filipovic2001}.
Fix $0\leq t\leq T\leq T^*$.
By the evolution equation for $X$, one has
\begin{align}
-\log P(t,T)
&=
X_t(0,T-t]
\\&=
S_t X_0(0,T-t]
+
\int_0^t S_{t-s}F_s(0,T-t] ds
+
\int_0^t S_{t-s}G_s(0,T-t] dW_s
\\&=
X_0(t,T]
+
\int_0^t F_s(t-s,T-s] ds
+
\int_0^t G_s(t-s,T-s] dW_s.
\end{align}
By the integral formula for $B_t$, the discounted bond price $Z(t,T)=P(t,T)/B_t$ satisfies
\begin{align}
-\log Z(t,T)
&=
X_0(0,T]
+
\int_0^t F_s(0,T-s] ds
+
\int_0^t G_s(0,T-s] dW_s.
\end{align}
This is a semimartingale, and It\^o's formula gives
\begin{multline}
Z(t,T)
=
Z(0,T)
-
\int_0^t Z(s,T)\Big(F_s(0,T-s]-\frac12 \|G_s(0,T-s]\|_U^2\Big) ds
\\-
\int_0^t Z(s,T) G_s(0,T-s] dW_s.
\end{multline}
Thus, $Z(\cdot,T)$ is a local martingale if and only if the first drift condition in statement \eqref{thm:term_structure_c} is satisfied.
A similar but simpler argument establishes the second drift condition.
\item Under the probability measure $\bP$, the value process of any admissible trading strategy in finitely many discounted bonds and discounted energy futures is a supermartingale, thanks to the Ansel--Stricker lemma.
Therefore, $\bP$ is a separating measure for the large financial market with discounted bonds and discounted energy futures as primary traded assets.
Consequently, no asymptotic free lunch with vanishing risk holds under $\bP$, and also under any measure equivalent to it \cite{cuchiero2016new}.
\end{enumerate}
\end{proof}

\begin{remark}[Bank account process]
The bank account process in \cref{thm:term_structure}~\eqref{thm:term_structure_b} coincides with the usual one for `regular' Heath--Jarrow--Morton models.
Here, regularity refers to requirement that the forward rate measures $X_t$ are non-cylindrical and have continuous Lebesgue densities $f_t$, which are subject to some additional regularity and integrability conditions \cite[Conditions H1--H3 and C1--C4]{Filipovic2001}.
Under these conditions, the bank account process may equivalently be defined by the well-known formula $B_t=\exp(\int_0^t r_s ds)$, where $r_s=f_s(0)$ is the short rate.
In particular, the bank account process is a semimartingale.
Moreover, by no-arbitrage considerations, the discounted bond prices, and consequently also the undiscounted ones, are semimartingales.
In contrast, beyond the regular case, the bank account process and the undiscounted bond price processes may not be semimartingales.
For instance, this happens if $X_t=S_t^*X_0$ is deterministic and $X_0$ has infinite variation.
\end{remark}

\begin{example}[Cylindrical term structure models]
\label{ex:term_structure}
It is easy to specify coefficients $F,\tilde F$ and $G,\tilde G$ which satisfy the assumptions and drift conditions of \cref{thm:cylindrical_stochastic_evolution,thm:term_structure}, as demonstrated in the following simple example.
Let $e\in L^\infty(\bR_{>0})^n$,
and let $g,\tilde g:\bR^n\to\cN_2(U,\cM(\bR_{>0}))$ be bounded Lipschitz functions,
where $\cN_2$ denotes 2-nuclear operators.
Then, for any $X \in L(C_0(\bR_{>0}),C_\bF([0,T],L^2(\Omega)))$, one has $X_t(e) \in L^2(\Omega,\bR^n)$ via the Riesz isomorphism in \cref{lem:random_measures}~\eqref{lem:random_measures_b}, and $g(X_t(e)) \in L^\infty(\Omega,\cN_2(U,\cM(\bR_{>0})))$ thanks to the boundedness of $g$.
Thus, one obtains Lipschitz functions
\begin{equation}
G,\tilde G:L(C_0(\bR_{>0}),C_\bF([0,T],L^2(\Omega)))\to L(C_0(\bR_{>0}),C_\bF([0,T],L^2(\Omega,U)))
\end{equation}
by identifying $U$ with $U^*$ and setting
\begin{equation}
G(X)(f)(t) = g(X_t(e))(f),
\qquad
\tilde G(X)(f)(t) = \tilde g(X_t(e))(f).
\end{equation}
Moreover, one obtains Lipschitz functions
\begin{equation}
F,\tilde F:L(C_0(\bR_{>0}),C_\bF([0,T],L^2(\Omega)))\to L(C_0(\bR_{>0}),C_\bF([0,T],L^2(\Omega))),
\end{equation}
which satisfy the drift conditions of \cref{thm:term_structure}, by the following definition, where the product of a measure $\mu\in\cM(\bR_{>0})$ with its distribution function $\int\mu\in \cL^\infty(\bR_{>0})$ is denoted by $\mu\int\mu\in\cM(\bR_{>0})$, and where $\Tr$ denotes the trace on $U$:
\begin{equation}
F(X)(f)(t) = \frac12 \Tr\big(g(X_t(e)){\textstyle\int} g(X_t(e))\big)(f),
\qquad
\tilde F(X) = 0.
\end{equation}
Alternative specifications are possible.
For instance, the coefficients could be time-dependent and could depend simultaneously on $X$ and $\tilde X$.
Note that $G$, as specified here, is non-cylindrically measure-valued.
We have not been able to construct $F$ subject to the drift condition when $G$ is merely cylindrically measure-valued.
Nevertheless, even for non-cylindrical $G$, a cylindrical setting is needed because the solution $X$ will be merely cylindrical, as noted in \cref{rem:measure_valued_integration}.
\end{example}

\begin{remark}[Measure-valued processes via martingale problems]
\label{rem:martingale_problems}
A well-trodden path for constructing measure-valued processes is via martingale problems.
Recently, this approach has been used to define models for discontinuous term structures \cite{cuchiero2022measure}, similarly to the ones defined here via cylindrical stochastic analysis.
The two approaches are closely related, as described next.
Let $X$ be a mild solution of the stochastic evolution equation in \cref{cor:cylindrically_evolution} with Markovian coefficients $F(X)_s=F(X_s)$ and $G(X)_s=G(X_s)$, i.e.,
\begin{equation}
X_t = S_t^*X_0 + \int_0^t S_{t-s}^*F(X_s) ds + \int_0^t S_{t-s}^*G(X_s) dW_s.
\end{equation}
Then, thanks to the (square) integrability of the coefficients,
$X$ is also a weak solution of the stochastic evolution equation \cite[Theorem~6.5]{daprato2014stochastic}, i.e., letting $A$ denote the generator of the shift semigroup $S$, one has for all $e \in D(A)$ that
\begin{equation}
X_t(e) = X_0(e) + \int_0^t X_s(Ae) ds + \int_0^t F(X_s)(e) ds + \int_0^t G(X_s)(e) dW_s.
\end{equation}
Thus, for any cylindrical functional $f(x)=\varphi(x(e_1),\dots,x(e_n))$, where $\varphi \in C_b^\infty(\bR^n,\bR)$, $n\in\bN$, and $e_1,\dots,e_n \in D(A)$, the process
\begin{multline}
f(X_t) - f(X_0) - \int_0^t Af'(X_s)X_s ds - \int_0^t f'(X_s)F(X_s) ds
\\
- \frac12 \int_0^t f''(X_s)(G(X_s),G(X_s)) ds
\end{multline}
is a local martingale, where $Af'$ is interpreted in the following sense:
\begin{equation}
Af'(x)y=g'(x(e_1),\dots,x(e_n))(y(Ae_1),\dots,y(Ae_n)).
\end{equation}
Assuming that $F$ and $G$ map non-cylindrical random measures to non-cylindrical random measures, as for instance in \cref{ex:term_structure}, it follows that $X$ solves the martingale problem with extended generator
\begin{align}
Lf(x) = Af'(x)x+f'(x)F(x)+\frac12 f''(x)(G(x),G(x)),
\qquad
x \in \cM(\bR_{>0}).
\end{align}
Under suitable conditions, existence and uniqueness of solutions to such martingale problems can be shown directly by Markovian methods.
Notably, this allows one to construct non-cylindrical processes with values in the cone of non-negative measures.
Of course, coefficients which lead to genuinely cylindrical processes, as for example in \cref{rem:measure_valued_integration}, are excluded.
Solutions of the martingale problem are, in great generality, again weak and mild solutions as above, for some Brownian motion reconstructed from the solution on an extended probability space \cite[Theorem~8.2]{daprato2014stochastic}.
\end{remark}

\appendix

\section{Cylindrical vector measures}
\label{sec:vector_measure}

This section explains cylindrification from the perspective of vector measures. 
Specifically, we show that \eqref{sec:vector_measure_a} cylindrical integrals in the sense of \cref{thm:cylindrical} are integrals with respect to cylindrical vector measures and conversely that \eqref{sec:vector_measure_b} cylindrical vector measures give rise to a cylindrical integral similar to \cref{thm:cylindrical}.
Moreover, \eqref{sec:vector_measure_c} cylindrical vector measures can be identified with operator-valued vector measures.
These relations are clearest in the case $I=\cL^\infty(S,F)$, which shall be assumed here.
\begin{enumerate}[(a)]
\item\label{sec:vector_measure_a}
The space $\cL^\infty(S,L(E,F))$ of non-cylindrical integrands is continuously included in the space $L(E,\cL^\infty(S,F))$ of cylindrical integrands.
Therefore, the cylindrical integral restricts to a continuous linear operator $\cL^\infty(S,L(E,F))\to L(E,H)$ or, equivalently, a continuous linear operator $E\to L(\cL^\infty(S,L(E,F)),H)$.
In other words, for any $e \in E$, the function $e\otimes\mu:\cA\to E\hat\otimes G$ belongs to the space $ba$ of finitely additive functions $\cA\to E \hat\otimes G$ with finite semivariation $\|\cdot\|_{L(E,F),E\hat\otimes G,H}$, where $\hat\otimes$ is the projective tensor product.
Consequently, $\mu$ (or rather the linear map $e\mapsto e\otimes \mu$) belongs to $L(E,ba)$.
In this sense, $\mu$ is a cylindrical vector measure, and the cylindrical integral is a continuous extension of its elementary integral.
\item\label{sec:vector_measure_b} Of course, not all elements of $L(E,ba)$ are of the simple form $e\mapsto e\otimes\mu$.
To treat the general case, let $\mu \in L(E,ba)$, where $ba$ now denotes the set of finitely additive functions $\cA\to G$ with finite semivariation $\|\cdot\|_{F,G,H}$. Under the mild assumption that for all $e \in E$, the elementary integral with respect to $\mu(e)$ extends continuously to $\cL^\infty(S,F)$, one obtains a continuous integral $\cL^\infty(S,F)\to L(E,H)$. Thus, starting from an $E$-cylindrical vector measure, one obtains an integral with values in the space $L(E,H)$, similarly to \cref{thm:cylindrical}.
\item\label{sec:vector_measure_c} Cylindrical vector measures $\mu:E\to (\cA\to G)$ in the sense of \eqref{sec:vector_measure_b} can be identified with operator-valued vector measures $\nu:\cA\to(E\to G)$ by a flip of coordinates: $\nu(A)(e)=\mu(e)(A)$.
This identification is a Banach isometry.
To see this, recall that the semivariation norm of a vector measure is the operator norm of its elementary integral and use the fact that $L(E,L(F,G))\cong L(F,L(E,G))$ for any Banach spaces $E$, $F$, and $G$.
Similar results for countably additive vector measures have been obtained in \cite[Theorem~4.2]{assefa2022cylindrical} building on \cite{graves1977theory}.
\end{enumerate}

\section{Measurable and predictable versions}
\label{sec:measurable_version}

\begin{lemma}[Measurable version]
\label{lem:measurable}
Let $T\in\bR_+$, let $(S,\Sigma,\mu)$ be a measure space, let $E$ be a Banach space,
and let $D([0,T],L^2(S,E))$ be the Banach space of c\`adl\`ag functions $[0,T]\to L^2(S,E)$ with the supremum norm.
Then, any process $X \in D([0,T],L^2(S,E))$ has a strongly measurable version $Y:[0,T]\times S\to E$, which is uniquely defined up to $dt\otimes\mu$-null sets, thereby giving rise to a linear embedding with operator norm bounded by $\sqrt{T}$:
\begin{equation}
D([0,T],L^2(S,E)) \ni X \mapsto Y \in L^2([0,T]\times S,E).
\end{equation}
\end{lemma}

\begin{proof}
We follow the lines of \cite[Proposition~3.2]{daprato2014stochastic} with continuity replaced by right-continuity with left limits.
Let $X \in D([0,T],L^2(S,E))$.
For any $n\in\bN$, there exists a finite set $A_n\subset[0,T]$, which contains zero and satisfies
\begin{equation}
\sup_{t\in[0,T]} \|X_t-X_{\lfloor t\rfloor_n}\|_{L^2(\Omega,E)}\leq 2^{-n},
\end{equation}
where $\lfloor t\rfloor_n=\max (A_n\cap [0,t])$ denotes rounding down to the grid $A_n$.
For any $t\in[0,T]$, let $N_t$ be the set of all $s\in S$ such that $X_{\lfloor t\rfloor_n}(s)$ does not converge to $X_t(s)$.
By Borell--Cantelli, $N_t$ is a null set.
The set $N=\bigcup_{t\in[0,T]}\{t\}\times N_t$ is measurable and a $dt\otimes\mu$-null set.
Define $Y:[0,T]\times\Omega\to E$ as $\lim_{n\to\infty}X_t(\lfloor s\rfloor_n)$ for $(t,s)\notin N$ and as zero otherwise.
Then, $Y$ is a version of $X$ and is strongly measurable as a pointwise limit of strongly measurable functions.
The $L^2$-norm of $Y$ is bounded by $\sqrt{T}$ times the supremum norm of $X$.
In particular, for any other strongly measurable version $\tilde Y$, the $L^2$ norm of $Y-\tilde Y$ vanishes, which implies that $Y$ is determined uniquely up to $dt\otimes\bP$-null sets.
\end{proof}

\begin{lemma}[Predictable version]
\label{lem:predictable}
Let $T\in\bR_+$, let $(\Omega,\cF,\bF=(\cF_t)_{t\in[0,T]},\bP)$ be a filtered probability space satisfying the usual conditions, let $E$ be a Banach space, let $C_\bF([0,T],L^2(\Omega,E))$ be the Banach space of adapted processes in $C([0,T],L^2(\Omega,E))$, and let $\Omega_T=[0,T]\times\Omega$ be endowed with the measure $dt\otimes\bP$ restricted to the predictable sigma algebra.
Then, any process $X \in C_\bF([0,T],L^2(\Omega,E))$ has a predictable version $Y:\Omega_T\to E$, which is uniquely defined up to $dt\otimes\bP$-null sets, thereby giving rise to a continuous linear embedding of operator norm bounded by $\sqrt{T}$:
\begin{equation}
C_\bF([0,T],L^2(\Omega)) \ni X \mapsto Y \in L^2(\Omega_T).
\end{equation}
\end{lemma}

\begin{proof}
See \cite[Proposition~3.7]{daprato2014stochastic} or \cite[Lemma~2.4]{tappe2010note} for the existence of a predictable version and \cite[Lemma~2.4 and Proposition~2.5]{tappe2010note} for the continuous linear embedding.
\end{proof}

\printbibliography

\end{document}